\documentclass[preprint, 10pt, english]{elsarticle}
\usepackage{amsthm}
\usepackage{amsmath}
\usepackage{latexsym, amssymb}
\usepackage{txfonts}
\usepackage{mathtools}
\usepackage{color}
\usepackage{babel}
\usepackage[all]{xy}
\usepackage[capitalise]{cleveref}

\newtheorem{thm}{Theorem}[section] 

\newtheorem{lem}[thm]{Lemma}

\theoremstyle{definition}
\newtheorem{rem}[thm]{Remark}

\newcommand\operA[2]{{\if!#2!\operatorname{#1}\else{\operatorname{#1}_{#2}^{\phantom{I}}}\fi}} 

\def\tr{{\operatorname{Tr}}}

\def\dim{{\operatorname{dim}}}

\newcommand{\Trace}[1][]{\if!#1!\operatorname{Tr}\else{\operatorname{Tr}_{#1}^{\phantom{I}}}\fi} 

\long\def\forget#1\forgotten{{}} %

\def\({\left(}
\def\){\right)}

\newcommand\LAY[3][]{{\begin{array}{c}\mbox{#2} \if#1!{}\else{+}\fi \\ \mbox{#3}\end{array}}}

\makeatletter

\def\ps@pprintTitle{%
 \let\@oddhead\@empty
 \let\@evenhead\@empty
 \def\@oddfoot{}%
 \let\@evenfoot\@oddfoot}

\newcommand{\bigperp}{%
  \mathop{\mathpalette\bigp@rp\relax}%
  \displaylimits
}

\newcommand{\bigp@rp}[2]{%
  \vcenter{
    \m@th\hbox{\scalebox{\ifx#1\displaystyle2.1\else1.5\fi}{$#1\perp$}}
  }%
}
\makeatother

\renewcommand{\geq}{\geqslant}
\renewcommand{\leq}{\leqslant}

\newif\iffurther
\furtherfalse

\journal{Pure and Applied Algebra}

\begin{document}
\begin{frontmatter}

\title{Linkage of Sets of Cyclic Algebras}

\author{Adam Chapman}
\ead{adam1chapman@yahoo.com}
\address{School of Computer Science, Academic College of Tel-Aviv-Yaffo, Rabenu Yeruham St., P.O.B 8401 Yaffo, 6818211, Israel}

\begin{abstract}
Let $p$ be a prime integer and $F$ the function field in two algebraically independent variables over a smaller field $F_0$.
We prove that if $\operatorname{char}(F_0)=p\geq 3$ then there exist $p^2-1$ cyclic algebras of degree $p$ over $F$ that have no maximal subfield in common, and if $\operatorname{char}(F_0)=0$ then there exist $p^2$ cyclic algebras of degree $p$ over $F$ that have no maximal subfield in common. 
\end{abstract}

\end{frontmatter}

\section{Introduction}

A cyclic algebra of prime degree $p$ over a field $F$ takes the form
$$(\alpha,\beta)_{p,F}=F \langle x,y : x^p=\alpha, y^p=\beta, y x y^{-1}=\rho x\rangle,$$
for some $\alpha,\beta \in F^\times$ when $\operatorname{char}(F)\neq p$ and $F$ contains a primitive $p$th root of unity $\rho$.
This algebra is a division algebra if $\alpha \not \in (F^\times)^p$ and $\beta$ is not a norm in the field extension $F[\sqrt[p]{\alpha}]/F$, and otherwise it is the matrix algebra $M_p(F)$.
When $\operatorname{char}(F)=p$, a cyclic algebra of degree $p$ over $F$ takes the form
$$[\alpha,\beta)_{p,F}=F \langle x,y : x^p-x=\alpha, y^p=\beta, y x y^{-1}=x+1 \rangle,$$
for some $\alpha \in F$ and $\beta \in F^\times$.
This algebra is a division algebra if $\alpha \not \in \wp(F)=\{\lambda^p-\lambda:\lambda\in F\}$ and $\beta$ is not a norm in the field extension $F[x: x^p-x=\alpha]/F$, and otherwise it is the matrix algebra $M_p(F)$.
These algebras won their significance for being the generators of ${_pBr}(F)$ (see \cite{MS} and \cite[Chapter 9]{GilleSzamuely:2006}). 
These algebras are called ``quaternion algebras" when $p=2$.

We say that cyclic algebras $A_1,\dots,A_\ell$ of degree $p$ over $F$ are linked if they share a common maximal subfield.
We say that ${_pBr}(F)$ is $\ell$-linked if every $\ell$ cyclic algebras of degree $p$ over $F$ are linked.

The linkage properties of such algebras demonstrate a deeper phenomenon yet to be fully understood: clearly if $A$ and $B$ are linked then $A \otimes B$ is not a division algebra, but for quaternion algebras the converse holds true as well.
This means that ${_2Br}(F)$ is 2-linked if and only if its symbol length is $\leq 1$ (i.e., every class is represented by a single quaternion algebra).
Moreover, if ${_2Br}(F)$ is 2-linked then the $u$-invariant of $F$ is either 0,1,2,4 or 8 (\cite{ElmanLam:1973} and \cite{ChapmanDolphin:2017}), and for nonreal fields $F$, ${_2Br}(F)$ is 3-linked if and only if $u(F) \leq 4$ (see \cite{Becher} and \cite{ChapmanDolphinLeep}).

For local fields $F$, ${_pBr}(F)$ is clearly $\ell$-linked for any $\ell$. It follows from the local-global principle (e.g., see \cite{ChapmanTignol:2019} and \cite[Proposition 15]{Sivatski:2014}) that for global fields $F$, ${_pBr}(F)$
 is $\ell$-linked for any $\ell$ too.
A question was raised (\cite{Becher}) on whether function fields $F=F_0(\alpha,\beta)$ in two algebraically independent variables over algebraically closed fields $F_0$ also satisfy this property.
It was answered in the negative for quaternion algebras (\cite{ChapmanTignol:2019} for $\operatorname{char}(F)=0$ and \cite{Chapman:2020} for $\operatorname{char}(F)=2$), showing that for such fields ${_2Br}(F)$ is not 4-linked.

In the current paper, we extend this observation to cyclic algebras of odd prime degree $p$ over $F=F_0(\alpha,\beta)$, showing that when $\operatorname{char}(F_0)=p$, the group ${_pBr}(F)$ is not $(p^2-1)$-linked, and when $\operatorname{char}(F_0)=0$, the group ${_pBr}(F)$ is not $p^2$-linked.

\section{Characteristic $p$}

\begin{lem}\label{Trace}
Let $A=[\alpha,\beta)_{p,F}$ be a cyclic algebra of degree $p$ generated by $x$ and $y$ over a field $F$ of $\operatorname{char}(F)=p$, and write $\tr : A \rightarrow F$ for its reduced trace map.
Then for any $\lambda=\sum_{i=0}^{p-1} \sum_{j=0}^{p-1} c_{i,j} x^i y^j \in A$, $\tr(\lambda)=-c_{p-1,0}$.
\end{lem}
\begin{proof}
For each $j\in \{1,\dots,p-1\}$, every element $v$ in $F(x)y^j$ satisfies $v^p \in F$, and so $\tr(v)=0$.
The problem therefore reduces to calculating $\tr_{K/F}(x^i)$
where $K = F(x)$ with $x$ a root of the irreducible polynomial $X^p-X-\alpha$ in $F[X]$. For this, let
$L_{x^i}: K \rightarrow K$ be the $F$-linear transformation given by multiplication by $x^i$. For $i \in \{1,2,\dots,p-1\}$,
the matrix $[L_{x^i}]$ of $L_{x^i}$ relative to the $F$-basis $\{1,x,\dots,x^{p-1}\}$ of $K$ has the following form: a diagonal of $1$-s starting at the $(i+1,1)$-entry (i.e., row $i+1$ and column 1), a diagonal of $\alpha$-s starting at the $(1,p- i + 1)$-entry, a $1$ directly below each $\alpha$, and all other entries are 0-s. Thus, $\tr_{K/F} (x^i) = \tr[L_{x^i}] = 0$ for $i \in \{1,\dots,p-2\}$,
while $\tr_{K/F} (x^{p-1}) = \tr[L_{x^{p-1}} ] = p -1$.
\end{proof}
\begin{rem}
The last statement appeared in {\cite[Remark 2.2]{ChapmanChapman:2017}}, but we provided here a simpler proof which was suggested by an anonymous colleague.
Note that the trace argument works in a more general setting, in any characteristic and  for roots $x$ of any irreducible polynomial $X^n-X-\alpha$ for any natural number $n$.
\end{rem}

\begin{thm}\label{generalp}
Let $p$ be an odd prime, $F_0$ a field of $\operatorname{char}(F_0)=p$ and $F=F_0(\alpha,\beta)$ the function field in two algebraically independent variables $\alpha$ and $\beta$ over $F_0$. Then there exist $p^2-1$ cyclic algebras of degree $p$ over $F$ that share no maximal subfield.
\end{thm}

\begin{proof}
Note that $F$ is endowed with the right-to-left $(\alpha^{-1},\beta^{-1})$-adic valuation, which we denote by $\mathfrak{v}$. This is in fact the restriction to the standard rank 2 valuation on $F_0(\!(\alpha^{-1})\!)(\!(\beta^{-1})\!)$. Write $\Gamma_F$ for the value group of $F$ with respect to $\mathfrak{v}$. Note $\Gamma_F=\mathbb{Z} \times \mathbb{Z}$.
For each $(i,j)\in I=\{0,1,\dots,p-1\}^{\times 2} \setminus \{(0,0)\}$, write 
$$A_{i,j}=\begin{cases}
[\alpha^i \beta^j,\beta)_{p,F} & i\neq 0\\
[\beta^j,\alpha)_{p,F} & i=0.
\end{cases}$$
Since the values of $\alpha^i \beta^j$ and $\beta$ when $i \neq 0$ are negative and $\mathbb{F}_p$-independent in $\Gamma_F/p\Gamma_F$, the valuation $\mathfrak{v}$ extends to $A_{i,j}$ and $A_{i,j}$ is totally ramified over $F$ with value group $\Gamma_{A_{i,j}}=\frac{1}{p} \mathbb{Z} \times \frac{1}{p} \mathbb{Z}$ (see \cite{Tignol:1992}).
For a similar argument, the valuation extends also to $A_{0,j}$, when $j \neq 0$, and $A_{0,j}$ is totally ramified over $F$ with value group $\Gamma_{A_{0,j}}=\frac{1}{p} \mathbb{Z} \times \frac{1}{p} \mathbb{Z}$ (see also \cite[Remark 3.2]{ChapmanChapman:2017}).
Write $V_{i,j}$ for the subspace of trace zero elements of $A_{i,j}$.
It follows from Lemma \ref{Trace} that $\mathfrak{v}(V_{i,j})/\Gamma_F \not \ni (\frac{i}{p},\frac{j}{p})$, because writing $x$ and $y$ for the standard generators of $A_{i,j}$, $V_{i,j}=\operatorname{Span}_F\{x^k y^\ell : (k,\ell)\in \{0,1,\dots,p-1\}^{\times 2} \setminus \{(p-1,0)\}\}$, the values of the $x^ky^\ell$-s are distinct modulo $\Gamma_F$ and none is congruent to $(\frac{i}{p},\frac{j}{p})$.
Therefore the intersection of all the $\mathfrak{v}(V_{i,j})$-s modulo $\Gamma_F$ is trivial, which means that $\bigcap_{(i,j)\in I} \mathfrak{v}(V_{i,j})=\Gamma_F$.

Now, suppose the contrary, that the algebras above share a maximal subfield $K$. Since $K$ is a subfield of each $A_{i,j}$, it is totally ramified over $F$.
Write $W$ for its subspace of elements of trace 0. Then $\dim_F W$ is at least $p-1$.
Since $W \subseteq  \bigcap_{(i,j)\in I} V_{i,j}$, the values of all the nonzero elements in $W$ are in $\Gamma_F$.
Recall that $p>2$. We can therefore choose two elements $w_1$ and $w_2$ in $K$ whose values are $\mathbb{F}_p$-independent in $\Gamma_K/\Gamma_F$. As a result, they are also linearly independent, and there is a nonzero linear combination of theirs $w_3 \in F w_1 + F w_2$ which lives in $W$. Hence, the value of $w_3$ is either $\mathfrak{v}(w_1)$ or $\mathfrak{v}(w_2)$. In either case, $\mathfrak{v}(w_3)$ is not in $\Gamma_F$, despite the fact that $w_3 \in W$, contradiction.
Consequently, the algebras $A_{i,j}$ have no maximal subfield in common.
\end{proof}

\begin{rem}
Theorem \ref{generalp} holds true also if one replaces $F_0(\alpha,\beta)$ with the field of iterated Laurent series $F_0(\!(\alpha^{-1})\!)(\!(\beta^{-1})\!)$ in two variables over $F_0$ of $\operatorname{char}(F_0)=p$.
This demonstrates another difference between the behaviour of Laurent series in the good characteristic and the bad characteristic, because over an algebraically closed field $F_0$ of $\operatorname{char}(F_0)=0$, the group ${_pBr}(F)$ for $F=F_0(\!(\alpha^{-1})\!)(\!(\beta^{-1})\!)$ is generated by a single division cyclic algebra of degree $p$ and thus ${_pBr}(F)$ is $\ell$-linked for any $\ell$.
\end{rem}

\section{Characteristic 0}

For the proof of the main result, we need the following observation about inseparable field extensions:
\begin{lem}\label{useful}
Let $E$ be a field of $\operatorname{char}(E)=p>0$, and $\alpha \in E \setminus E^p$.
\begin{enumerate}
\item Then the $E^p$-vector spaces $V_i=\operatorname{Span}_{E^p}\{(\alpha-i)^k : k \in \{1,\dots,p-1\}\}$ for $i \in \{0,\dots,p-1\}$ satisfy $\bigcap_{i=0}^{p-1} V_i=\{0\}$.
\item Furthermore, if there is another element $\beta \in E \setminus E^p(\alpha)$, then the vector spaces $W_{i,j}$ given by $W_{i,0}=\operatorname{Span}_{E^p}\{(\alpha-i)^m\beta^n : m,n \in \{0,\dots,p-1\}, (m,n)\neq (0,0)\}$ and $W_{i,j}=\operatorname{Span}_{E^p}\{\alpha^m (\alpha^i\beta-j)^n : m,n \in \{0,\dots,p-1\}, (m,n)\neq (0,0)\}$ for $i\in \{0,\dots,p-1\}$ and $j \in \{1,\dots,p-1\}$, satisfy $\bigcap_{i,j=0}^{p-1} W_{i,j}=\{0\}$.
\end{enumerate}
\end{lem}

\begin{proof}
The first statement follows from the fact that an element $v=c_0+c_1 \alpha+\dots+c_{p-1} \alpha^{p-1} \in E^p(\alpha)$ is in $V_i$ if and only if 
$$c_0+c_1 i+c_2 i^2+\dots+c_{p-1} i^{p-1}=0.$$
If we assume that $v \in \bigcap_{i=0}^{p-1}V_i$, then the polynomial $c_0+c_1 X+\dots+c_{p-1} X^{p-1}\in E^p[X]$ has at least $p$ distinct roots in $E^p$ (which are the elements of the subfield $\mathbb{F}_p$), it must be the zero polynomial, i.e., $c_0=c_1=\dots=c_{p-1}=0$.

For the second statement, we first note that 
$W_{i,0}$ can be written as 
$$W_{i,0}=\operatorname{Span}_{E^p(\alpha-i)}\{\beta^k : k \in \{1,\dots,p-1\}\}\oplus\operatorname{Span}_{E^p}\{(\alpha-i)^k : k \in \{1,\dots,p-1\}\}=$$
$$=\operatorname{Span}_{E^p(\alpha)}\{\beta^k : k \in \{1,\dots,p-1\}\}\oplus\operatorname{Span}_{E^p}\{(\alpha-i)^k : k \in \{1,\dots,p-1\}\}.$$
It follows from the first statement that
$$\bigcap_{i=0}^{p-1} W_{i,0}=\operatorname{Span}_{E^p(\alpha)}\{\beta^k : k \in \{1,\dots,p-1\}\}.$$
Now, for any $i \in \{0,\dots,p-1\}$, the space $W_{0,0}$ can also be written as
$$W_{0,0}=\operatorname{Span}_{E^p(\alpha^i \beta)}\{\alpha^k : k \in \{1,\dots,p-1\}\}\oplus\operatorname{Span}_{E^p}\{(\alpha^i \beta-0)^k : k \in \{1,\dots,p-1\}\},$$
and for any $j \in \{1,\dots,p-1\}$ we can write the space $W_{i,j}$ as
$$W_{i,j}=\operatorname{Span}_{E^p(\alpha^i \beta-j)}\{\alpha^k : k \in \{1,\dots,p-1\}\}\oplus\operatorname{Span}_{E^p}\{(\alpha^i \beta-j)^k : k \in \{1,\dots,p-1\}\}$$
$$=\operatorname{Span}_{E^p(\alpha^i \beta)}\{\alpha^k : k \in \{1,\dots,p-1\}\}\oplus\operatorname{Span}_{E^p}\{(\alpha^i \beta-j)^k : k \in \{1,\dots,p-1\}\}.$$
It follows from the first statement that for any $i \in \{0,\dots,p-1\}$,
$$W_{0,0} \cap \left(\bigcap_{j=1}^{p-1} W_{i,j}\right)=\operatorname{Span}_{E^p(\alpha^i \beta)}\{\alpha^k : k \in \{1,\dots,p-1\}\}.$$
Since the intersection $\bigcap_{(i,j)\in \{0,\dots,p-1\}^{\times 2}} W_{i,j}$ can be written as
$$\bigcap_{(i,j)\in \{0,\dots,p-1\}^{\times 2}} W_{i,j}=\left( \bigcap_{i=0}^{p-1} W_{i,0} \right) \cap \left( \bigcap_{i=1}^{p-1}\left( W_{0,0} \cap \left(\bigcap_{j=1}^{p-1} W_{i,j}\right) \right) \right),$$
we conclude that
$$\bigcap_{(i,j)\in \{0,\dots,p-1\}^{\times 2}} W_{i,j}=\operatorname{Span}_{E^p(\alpha)}\{\beta^k : k \in \{1,\dots,p-1\}\} \cap \left( \bigcap_{i=1}^{p-1}\left( \operatorname{Span}_{E^p(\alpha^i \beta)}\{\alpha^k : k \in \{1,\dots,p-1\}\} \right) \right),$$
and the intersection on the right-hand side is clearly trivial.
\end{proof}

\begin{thm}
Let $F_0$ be a field of $\operatorname{char}(F_0)=0$ and $F=F_0(\alpha,\beta)$ the function field in two algebraically independent variables over $F_0$.
Then there exist $p^2$ cyclic algebras of degree $p$ over $F$ that have no maximal subfield in common.
\end{thm}

\begin{proof}
Consider the algebras $A_{i,j}=(\gamma_{i,j},\delta_{i,j})_{p,F}$ for $i,j \in \{0,\dots,p-1\}$ where $(\gamma_{i,j},\delta_{i,j})$ are given (as elements of $F\times F$) by the formula
$$(\gamma_{i,j},\delta_{i,j})=\begin{cases}
(\alpha-i,\beta) & (i,j) \in \{0,\dots,p-1\} \times \{0\} \\
(\alpha^i\beta-j,\alpha) &  (i,j) \in \{0,\dots,p-1\} \times \{1,\dots,p-1\}
\end{cases}$$
In the rest of the proof we can assume that $F_0$ is algebraically closed. If it is not, we can extend scalars to $F_0^{alg}(\alpha,\beta)$. If the algebras do not have a common maximal subfield under this restriction, they did not have any common maximal subfield from the beginning.
Denote by $V_{i,j}$ the subspace of $A_{i,j}$ of elements of trace zero.
Every maximal subfield is generated by an element of trace zero, and therefore in order for the algebras to have a common maximal subfield, they must posses nonzero elements of trace zero of the same reduced norm.
Write $\varphi_{i,j}$ for the restriction of the reduced norm to $V_{i,j}$, and thus a necessary condition for the algebras to share a maximal subfield is that the forms $\varphi_{i,j}$ for $i,j \in \{0,\dots,p-1\}$ represent a common nonzero value.
Now, the $p$-adic valuation extends from $\mathbb{Q}$ to $F_0$ with residue field $k$, and thus to $F$ with residue field $E=k(\alpha,\beta)$. (See \cite[Chapter 3]{EnglerPrestel:2005} for details.)
Since the value group of $F_0$ is divisible, if the forms represent a common nonzero value, we can suppose the equality $\varphi_{1,1}(v_{1,1})=\dots=\varphi_{p-1,p-1}(v_{p-1,p-1})$ is obtained for elements $v_{1,1},\dots,v_{p-1,p-1}$ of minimal value 0.
If such a solution to the system above exists, then it gives rise to a solution to the system 
\begin{eqnarray}\label{eq}
\overline{\varphi_{1,1}(v_{1,1})}=\dots=\overline{\varphi_{p-1,p-1}(v_{p-1,p-1})}
\end{eqnarray}
and their value in $k(\alpha,\beta)$ is nonzero as the residue of an element of value zero.

\sloppy The valuation extends from $F$ to the algebras $A_{i,j}$, which are unramified, and their residue algebras are $k(\sqrt[p]{\alpha},\sqrt[p]{\beta})$. (This follows from \cite[Proposition 3.38]{TignolWadsworth:2015} and the fact that $k(\sqrt[p]{\gamma_{i,j}},\sqrt[p]{\delta})=k(\sqrt[p]{\alpha},\sqrt[p]{\beta})$ is a degree $p^2$ field extension of $k$.)
Fixing generators $x$ and $y$ for $A_{i,j}$, the image of the reduced norm of an element $t=\sum_{m=0}^p \sum_{n=0}^p c_{m,n}x^my^n$ of value zero in the residue field $k(\alpha,\beta)$ is $\overline{t}^p$ (by \cite[Lemma 11.16]{TignolWadsworth:2015}), which is $\sum_{m=0}^p \sum_{n=0}^p c_{m,n}^p \gamma_{i,j}^m\delta_{i,j}^n$.
If $\tr(t)=0$ then $c_{0,0}=0$.
Thus, the solution to the system \ref{eq} gives rise to a nontrivial intersection of the $E^p$-vector spaces $W_{i,j}=\operatorname{Span}\{\gamma_{i,j}^m\delta_{i,j}^n: m,n \in \{0,\dots,p-1\}, (m,n)\neq (0,0)\}$ for $i,j \in \{0,\dots,p-1\}$.
However, they intersect trivially by Lemma \ref{useful}.
Hence, the algebras $A_{i,j}$ for $i,j \in \{0,\dots,p-1\}$ share no maximal subfield.
\end{proof}

\section{In the opposite direction}

It is important to point out what is known about the linkage of ${_pBr}(F)$ for function fields $F=F_0(\alpha,\beta)$ over algebraically closed fields $F_0$:
\begin{itemize}
\item When $p=2$, ${_2Br}(F)$ is 3-linked in any characteristic, so the story is complete in this case, for previous papers have shown that it need not be 4-linked.
\item When $p=3$, ${_3Br}(F)$ is 2-linked by an easy argument mentioned in \cite{Artin:1982} based on \cite{Lang:1952}. Here we show that it need not be 8-linked in characteristic 3, or 9-linked in characteristic 0. Between 2 and 8 or 9 there is still a significant gap.
\item There are no results in this direction for $p>3$ to the author's knowledge. There are results on the related period-index problem but that does not settle the problem yet.
\end{itemize}

\section{Acknowledgements}

The author thanks Jean-Pierre Tignol for his helpful comments on the manuscript.

\def\cprime{$'$}

\end{document}